\documentclass[12pt,leqno]{article}
\usepackage[doc]{optional}
\usepackage{color}
\usepackage{float}
\usepackage{soul}
\usepackage{graphicx}
\definecolor{labelkey}{rgb}{0,0.08,0.45}
\definecolor{refkey}{rgb}{0,0.6,0.0}
\definecolor{Brown}{rgb}{0.45,0.0,0.05}
\definecolor{lime}{rgb}{0.00,0.8,0.0}
\definecolor{lblue}{rgb}{0.5,0.5,0.99}

\usepackage{mathpazo}

\usepackage{amsmath}

\usepackage{stmaryrd}
\usepackage{amssymb}
\usepackage{theorem}

\usepackage{hyperref}

\newcommand{\sepp}{\setlength{\itemsep}{-2pt}}

\newcommand{\nnn}{\ensuremath{{n\in{\mathbb N}}}}

\newcommand{\menge}[2]{\big\{{#1}~\big |~{#2}\big\}}

\newcommand{\fenv}[1]%
{\ensuremath{\,\overrightarrow{\operatorname{env}}_{#1}}}
\newcommand{\benv}[1]%
{\ensuremath{\,\overleftarrow{\operatorname{env}}_{#1}}}

\newcommand{\ve}{\ensuremath{\varepsilon}}

\newcommand{\exi}{\ensuremath{\exists\,}}

\newcommand{\RR}{\ensuremath{\mathbb R}}
\newcommand{\RP}{\ensuremath{\mathbb{R}_+}}
\newcommand{\RPP}{\ensuremath{\mathbb{R}_{++}}}

\newcommand{\NN}{\ensuremath{\mathbb N}}

\newcommand{\pinf}{\ensuremath{+\infty}}

{\begin{list}{}{%
\settowidth{\labelwidth}{\textrm{#1~}}%
\setlength{\leftmargin}{\labelwidth+\labelsep}}}
{\end{list}}
\newtheorem{theorem}{Theorem}[section]

\newtheorem{corollary}[theorem]{Corollary}

\theoremstyle{plain}{\theorembodyfont{\rmfamily}
}
\theoremstyle{plain}{\theorembodyfont{\rmfamily}
}
\theoremstyle{plain}{\theorembodyfont{\rmfamily}
}
\theoremstyle{plain}{\theorembodyfont{\rmfamily}
}

\theoremstyle{plain}{\theorembodyfont{\rmfamily}
\newtheorem{remark}[theorem]{Remark}}

\allowdisplaybreaks 

\begin{document}

\title{{On cluster points of
alternating projections}}

\author{
Heinz H.\ Bauschke\thanks{
Mathematics, University
of British Columbia,
Kelowna, B.C.\ V1V~1V7, Canada. E-mail:
\texttt{heinz.bauschke@ubc.ca}.}~~~and~
Dominikus Noll\thanks{
Universit{\'e} Paul Sabatier,
Institut de Math{\'e}matiques,
118 route de Narbonne, 31062 Toulouse, France.
E-mail: \texttt{noll@mip.ups-tlse.fr}.}
}

\date{July 9, 2013}
\maketitle

\vskip 8mm

\begin{abstract} \noindent
Suppose that $A$ and $B$ are closed subsets of a
Euclidean space such that $A\cap B\neq\varnothing$, 
and we aim to find a point in this intersection
with the help of 
the sequences $(a_n)_\nnn$ and $(b_n)_\nnn$ generated by the 
\emph{method of alternating projections}.
It is well known that if $A$ and $B$ are convex, then 
$(a_n)_\nnn$ and $(b_n)_\nnn$ converge to some point in $A\cap
B$. The situation in the nonconvex case is much more delicate.
In 1990, Combettes and Trussell presented a dichotomy result
that guarantees either convergence to a point in the intersection
or a nondegenerate compact continuum as the set of cluster points.

In this note, we construct two sets in the Euclidean plane
illustrating the continuum case. The sets $A$ and $B$ can be
chosen as countably infinite 
unions of closed convex sets. In contrast,
we also show that such behaviour is impossible for finite unions. 
\end{abstract}

{\small

}

\section{Motivation}

Let $X$ be a real Euclidean space, and let $A$ and $B$ be
closed subsets of $X$. Our aim is to find a point in 
$A \cap B$
which we assume to be nonempty. One classical algorithm is the
\emph{method of alternating projections}:
Given a starting point $b_{-1}\in X$, generate sequences
\begin{equation}
(\forall\nnn) \quad
a_{n} \in P_A(b_{n-1})
\;\;\text{and}\;\;
b_n \in P_B(a_n)
\end{equation}
where $P_Cx := \menge{c\in C}{\|x-c\| = d_C(x) := \inf_{y\in
C}\|x-y\|}$ denotes the \emph{projection} of $x$ onto $C$. 
When $A$ and $B$ are convex, then the projectors $P_A$ and $P_B$
are single-valued and the sequences $(a_n)_\nnn$ and $(b_n)_\nnn$
converge to some point in $A \cap B$.
This classical result goes back to Bregman \cite{Bregman}, and it has found
a huge number of extensions (see, e.g., \cite{BC2011}, \cite{CZ},
\cite{GK}, \cite{GR}).
In the general case, when $A$ and $B$ are not necessarily convex,
the situation is much more delicate.
In their 1990 paper \cite{CT}, 
Combettes and Trussell gave quite general 
sufficient conditions for the following dichotomy:
either $(a_n)_\nnn$ and $(b_n)_\nnn$ converge to a point in
$A\cap B$ or the set of cluster points is a nondegenerate
continuum. (For recent results in the nonconvex case, see \cite{BLPW1}
and \cite{BLPW2} and the references therein.)

\emph{The goal of this note is to explicitly construct two sets
$A$ and $B$ illustrating the continuum case.}

The sets $A$ and $B$ may be chosen to be 
countably infinite unions of closed convex sets.
In contrast, we also prove that the continuum case cannot occur
when $A$ and $B$ are finite unions of closed convex sets.

The remainder of the paper is organized as follows.
In Section~\ref{s:curve}, we lay the ground work by studying a
certain curve in the Euclidean plane.
In Section~\ref{s:main}, we use this curve to construct a
sequence of points in the plane that is crucial in obtaining the
sets $A$ and $B$. Some remarks and the announced positive result
conclude the paper.

\section{An intriguing curve}

\label{s:curve}

We will mostly work in the 
Euclidean plane $\RR^2$. 
As usual, angles will be measured in radians, but sometimes we
shall use degrees as in writing $\pi/2 = 90^\circ$. 

Let us recall that 
the distance $d$ between $(r\cos(\alpha),r\sin(\alpha))$ and 
$(s\cos(\beta),s\sin(\beta))$, where $r\in\RP$ and $\alpha\in\RR$,
satisfies
\begin{subequations}
\begin{align}
\label{e:130628c}
d^2 &= \|(r\cos(\alpha),r\sin(\alpha))-(s\cos(\beta),s\sin(\beta))\|^2= r^2 + s^2 -
2rs\cos(\alpha-\beta)\\
&\geq  r^2 + s^2 - 2rs 
=(r-s)^2; 
\end{align}
\end{subequations}
hence, 
\begin{equation}
\label{e:130628a}
r-d \leq s \leq r+d.
\end{equation}
Define the function $\rho$ by 
\begin{equation}
\rho\colon \RP \to \RP \colon t \mapsto 1+\exp(-t).
\end{equation}
This function will represent the distance of a point on 
the curve at time $t$ to the origin. 
Clearly, $\rho$ is strictly decreasing with $\rho(0)=2$ and 
$\lim_{t\to\pinf}\rho(t)=1$. 
Also define
\begin{equation}
\ve \colon \RP \to \RPP \colon
t\mapsto \frac{\rho(t)-\rho(t+2\pi)}{2}.
\end{equation}
Then $\ve' = -\ve$ and hence $\ve$ is strictly decreasing
to $\lim_{t\to\pinf}\ve(t)=0$. 
Note that
\begin{equation}
\label{e:130628d}
\RP\to\RPP\colon \alpha\mapsto \frac{\ve(\alpha)}{\rho(\alpha)} =
\frac{1}{2}\frac{1-e^{-2\pi}}{1+e^\alpha}
\text{~~is strictly decreasing.}
\end{equation}
We now define the curve
\begin{equation}
x\colon \RP\to\RR^2\colon \alpha\mapsto
\rho(\alpha)\cdot\big(\cos(\alpha),\sin(\alpha)\big).
\end{equation}
Note that $x$ describes a spiral traversing counter-clockwise;
$x$ is \emph{injective} because $\rho$ is strictly
decreasing. 
Now let $\alpha$ and $\beta$ be in $\RP$,
and assume that $\|x(\alpha)-x(\beta)\|\leq \ve(\alpha)$. 
By \eqref{e:130628a}, 
$\rho(\alpha)-\ve(\alpha)\leq \rho(\beta)\leq \rho(\alpha)+\ve(\alpha)$. 
Using the definitions, 
we solve these inequality for $\beta$ and obtain
\begin{equation}
\label{e:130628b}
\alpha-0.40 \approx \alpha +\ln(2)-\ln(3-e^{-2\pi}) \leq \beta \leq
\alpha+\ln(2)-\ln(1+e^{-2\pi})\approx\alpha+0.69;
\end{equation}
in degrees, this implies
$\alpha-24^\circ \leq \beta \leq \alpha+40^\circ$.
To summarize,
\begin{equation}
\label{e:130704a}
\|x(\alpha)-x(\beta)\|\leq \ve(\alpha)
\quad\Rightarrow\quad
\alpha-24^\circ \leq \beta \leq \alpha+40^\circ.
\end{equation}
We will now discuss the monotonicity of the function
\begin{equation}
f\colon t \mapsto \|x(\alpha+t)-x(\alpha)\|^2. 
\end{equation}
Because of the triangle inequality 
(or since $\sin(t)+\cos(t)=\sqrt{2}\sin(t+\pi/4)$), 
it is clear that 
\begin{equation}
t\in \left]0,\pi/2\right[
\quad\Rightarrow\quad
\sin(t)+\cos(t) > 1.
\end{equation}
One checks that
\begin{equation}
f'(t)\frac{\exp(2(\alpha+t))}{2} = 
g_1(t)+g_2(t)+g_3(t),
\end{equation}
where
\begin{subequations}
\begin{align}
g_1(t) &= \sin(t)\exp(2t+\alpha)(1+\exp(\alpha)),\\
g_2(t) &= \exp(\alpha+t)\big(\sin(t)+\cos(t)-1\big),\\
\quad
g_3(t)&= \exp(t)\big(\sin(t)+\cos(t)-\exp(-t)\big). 
\end{align}
\end{subequations}
Since each $g_i$ is strictly positive on $\left]0,\pi/2\right[$,
it follows from the mean value theorem that 
\begin{equation}
\label{e:130704b-}
\text{
$f$ is strictly increasing on $[0,\pi/2]$. 
}
\end{equation}
Combining with \eqref{e:130704a}, we deduce\footnote{``$\exi!\,$''
stands for ``there exists a \emph{unique}''}
\begin{equation}
\label{e:130704b}
\big(\forall \alpha\in\RP\big)\big(\exi!\,
\beta>\alpha\big)\;\;
\|x(\beta)-x(\alpha)\|=\ve(\alpha).
\end{equation}
Furthermore, denoting the unit sphere by
$S$, we have 
\begin{equation}
\label{e:130704f}
(\forall\alpha\in\RP)\quad
d_{S}(x(\alpha))=\rho(\alpha)-1 = \exp(-\alpha)>\ve(\alpha).
\end{equation}

\section{An intriguing sequence}

\label{s:main}

We now construct a sequence $(x_n)_\nnn$ in the Euclidean plane
with remarkable properties.
Let us initialize 
\begin{equation}
\alpha_0 := 0,\quad x_0 := x(\alpha_0), 
\quad \rho_0 := \rho(\alpha_0), \quad \ve_0 := \ve(\alpha_0).
\end{equation}
In Cartesian coordinates, $x_0=(2,0)$, and $\ve_0\approx
0.5$. 
Now suppose $\nnn$ and $\alpha_n$, $x_n$, $\rho_n$, and $\ve_n$
are given. 
In view of \eqref{e:130704b}, there exists a unique
$\beta>\alpha_n$ such that 
\begin{equation}
\|x(\beta)-x(\alpha_n)\|=\ve_n.
\end{equation}
We then update
\begin{equation}
\text{
$\alpha_{n+1}:=\beta$, $x_{n+1} := x(\alpha_{n+1})$,
$\rho_{n+1}:=\rho(\alpha_{n+1})$, and
$\ve_{n+1}:=\ve(\alpha_{n+1})$. 
}
\end{equation}
(The picture illustrates the beginning of the spiral and
$x_0,\ldots,x_{15}$ along with the radii used to construct the
next iterate.)
\begin{center}
\includegraphics[scale=0.5]{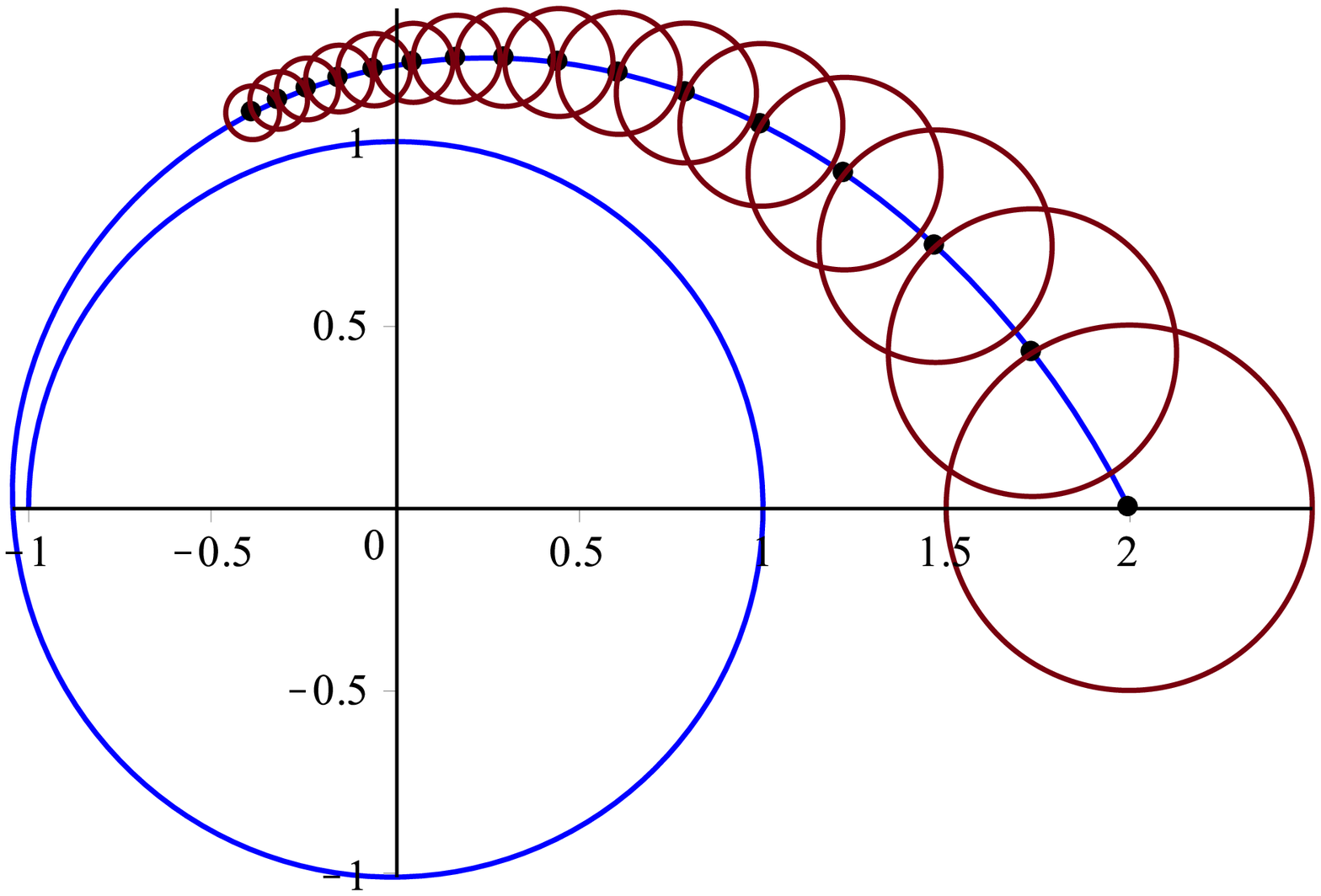}
\end{center}
We also set
\begin{equation}
\delta_{n} := \alpha_{n+1}-\alpha_n.
\end{equation}
By construction,
\begin{equation}
\label{e:130704c}
(\forall\nnn)\quad
\|x_n-x_{n+1}\|=\ve_n
\quad\text{and}\quad
\sum_{k=0}^{n} \delta_k= \alpha_{n+1}-\alpha_0.
\end{equation}
Note that 
\begin{equation}
(\alpha_n)_\nnn \text{~is strictly increasing, and }
(\ve_n)_\nnn \text{~is strictly decreasing}
\end{equation}
because the function $\ve$ is strictly decreasing.
Set 
\begin{equation}
\alpha_\infty := \lim_{\nnn}\alpha_n \in \left]0,\pinf\right].
\end{equation}
Since $\rho$ is strictly decreasing we also note that 
\begin{equation}
(\rho_n)_\nnn \text{~is strictly decreasing, with}
\lim_{\nnn}\rho_n =: \rho_\infty \in \left[1,2\right[.
\end{equation}
Hence the corresponding sequence of quotients satisfies
\begin{equation}
1 > q_n := \frac{\rho_{n+1}}{\rho_n} \to 1.
\end{equation}
Using \eqref{e:130628c} and the half-angle identity for sine, we have
\begin{subequations}
\label{e:130628f}
\begin{align}
(\forall\nnn)\quad 
\ve_n^2 &= \|x_n-x_{n+1}\|^2\\
&= \rho_n^2 + \rho_{n+1}^2 - 2\rho_n\rho_{n+1}\cos(\delta_n)\\
&= (\rho_n-\rho_{n+1})^2 + 2\rho_n\rho_{n+1}
\big(1-\cos(\delta_n)\big)\\
&= (\rho_n-\rho_{n+1})^2 + 4\rho_n\rho_{n+1}
\frac{1-\cos(\delta_n)}{2}\\
&= (\rho_n-\rho_{n+1})^2 + 4\rho_n\rho_{n+1}\sin^2(\delta_n/2).
\end{align}
\end{subequations}
Dividing by $\rho_n^2$ and recalling \eqref{e:130628d},
we obtain 
\begin{equation}
(\forall\nnn)\quad
\left(\frac{1}{2}\frac{1-e^{-2\pi}}{1+e^{\alpha_n}}\right)^2
= \frac{\varepsilon_n^2}{\rho_n^2} = (1-q_n)^2 +
4q_n\sin^2(\delta_n/2).
\end{equation}
Taking limits, we learn that
\begin{equation}
\label{e:130628e}
\left(\frac{1}{2}\frac{1-e^{-2\pi}}{1+e^{\alpha_\infty}}\right)^2
= 4\lim_{n}\sin^2(\delta_n/2).
\end{equation}
Since $\delta_n$, in degrees, belongs to
$\left]0^\circ,40^\circ\right]$
by \eqref{e:130704a}, we deduce that  $(\delta_n)_\nnn$ 
is convergent as well. 
If $\alpha_\infty=\pinf$, then $\delta_n\to 0$ by
\eqref{e:130628e}; however, 
if $\alpha_\infty<\pinf$, then $\delta_n =
\alpha_{n+1}-\alpha_n\to \alpha_\infty-\alpha_\infty = 0$.
Hence, we \emph{always} must have
\begin{equation}
\label{e:130708a}
\delta_n\to 0.
\end{equation}
Again by \eqref{e:130628e}, we have
\begin{equation}
\label{e:130704d-}
\alpha_n\to\alpha_\infty = \pinf, 
\end{equation}
which by \eqref{e:130704c} implies 
\begin{equation}
\label{e:130704d}
\sum_{\nnn} \delta_n = \pinf,
\end{equation}
\begin{equation}
\ve_n\to 0,
\end{equation}
and
\begin{equation}
\label{e:130704g}
\rho_n\to \rho_\infty = 1.
\end{equation}
Note also that in view of \eqref{e:130628f}, 
we have 
\begin{equation}
\ve_n^2 > 4\sin^2(\delta_n/2) \geq
\frac{\delta_n^2}{4}\quad\text{eventually}, 
\end{equation}
where we used \eqref{e:130708a} and the Taylor estimate
\begin{equation}
\sin(t/2) \geq \frac{1}{2}t - \frac{1}{48}t^3 
= \frac{t}{2}\Big(1 - \frac{1}{24}t^2\Big)\geq \frac{t}{4} 
\quad\text{for $t$ sufficiently close to $0$.}
\end{equation}
Combining with \eqref{e:130704d}, we record that
\begin{equation}
(\forall\nnn)\quad
\|x_{n}-x_{n+1}\|>\|x_{n+1}-x_{n+2}\|\to 0,
\quad\text{and}\quad
\sum_{\nnn}\|x_n-x_{n+1}\| = \pinf. 
\end{equation}
Furthermore, \eqref{e:130704d-} and \eqref{e:130704g} imply that 
\begin{equation}
\text{the set of cluster points of $(x_n)_\nnn$ is the unit
sphere $S$.}
\end{equation}
Define
\begin{equation}
(\forall\nnn)\quad
C_n := \{x_0,x_1,\ldots\}\smallsetminus\{x_n\}
\end{equation}
We claim that
\begin{equation}
\label{e:130704e}
(\forall\nnn)\quad
P_{C_n}x_n = \{x_{n+1}\}.
\end{equation}
Let $\nnn$. Since
$D_n := \{x_{n+1},x_{n+2},\ldots\}\subset
x\big(\left]\alpha_n,\pinf\right[\big)$, it follows from
\eqref{e:130704a}, \eqref{e:130704b-}, and 
\eqref{e:130704b} that $P_{D_n}x_n = \{x_{n+1}\}$. 
We show that there is no $k\in\NN$ such that
$k<n$ and $\|x_k-x_n\|<\|x_n-x_{n+1}\|$. 
Suppose the contrary. Then, 
by \eqref{e:130704a}, 
$\alpha_n-24^\circ \leq \alpha_k<\alpha_n$. 
Hence $\alpha_k<\alpha_n\leq\alpha_{k}+24^\circ$. 
By \eqref{e:130704b-}, 
$\|x_k-x_{k+1}\| = \|x(\alpha_k)-x(\alpha_{k+1})\|
\leq\|x(\alpha_k)-x(\alpha_n)\|=\|x_k-x_n\|<\|x_n-x_{n+1}\|<\|x_k-x_{k+1}\|$,
which is absurd. This verifies \eqref{e:130704e}. 
Furthermore, by \eqref{e:130704f}, 
\begin{equation}
(\forall\nnn)\quad d_S(x_n)>\|x_n-x_{n+1}\|.
\end{equation}

Let us summarize our findings.

\begin{theorem}
The sequence $(x_n)_\nnn$ and the set $Y := \menge{x_n}{\nnn}$ 
satisfy the following:
\begin{enumerate}
\item $(\|x_n-x_{n+1}\|)_\nnn$ is strictly decreasing.
\item $x_n-x_{n+1}\to 0$.
\item $\sum_\nnn \|x_n-x_{n+1}\|=\pinf$. 
\item $(\forall\nnn)$ $P_{(S\cup Y)\smallsetminus\{x_n\}}x_n = \{x_{n+1}\}$. 
\item The set of cluster points of $(x_n)_\nnn$ is the compact 
continuum $S$. 
\item $S\cup D$ is closed, where $D$ is an arbitrary subset of
$Y$. 
\end{enumerate}
\end{theorem}

We now obtain the announced example concerning an instance of the
method of alternating projections whose set of cluster points is
a nondegenerate compact continuum.

\begin{corollary}
\label{c:main}
Set $A := \menge{x_{2n}}{\nnn} \cup S$, $B :=
\menge{x_{2n+1}}{\nnn}\cup S$, and $b_{-1}:=x_0$. 
Then $A$ and $B$ are nonempty compact subsets of $\RR^2$. 
The corresponding sequences of alternating projections satisfy
\begin{equation}
(\forall\nnn)\quad
a_{n}=P_Ab_{n-1}=x_{2n}
\;\;\text{and}\;\;
b_n = P_Ba_n = x_{2n+1}.
\end{equation}
Moreover, $a_n-b_{n-1}\to 0$, $b_n-a_n\to 0$,
and $S$ is the set of cluster points of $(a_n)_\nnn$ and of 
$(b_n)_\nnn$. 
\end{corollary}

\begin{remark} Some comments on Corollary~\ref{c:main} are in
order. 
\begin{enumerate}
\item We note that Corollary~\ref{c:main}
is the first example constructed where the set of limit points of
alternating projections is a nondegenerate
compact continuum. This complements the analysis of Combettes and
Trussell \cite{CT} who conceived this case. 
\item 
If the starting point $b_{-1}$ is an arbitrary point,
then either $a_0\in S$ or $a_0\in A \smallsetminus S$.
In the first case, we have $(\forall\nnn)$ $a_n=b_n=a_0$; in the
second case, the sequences $(a_n)_\nnn$ and $(b_n)_\nnn$ are
tails of $(x_{2n})_\nnn$ and $(x_{2n+1})_\nnn$
respectively. A more involved analysis shows that if 
$b_{-1}$ is outside the closed unit disk, 
then $P_Ab_{-1}\in A\smallsetminus S$
and we are in the second case. Hence one obtains a nondegenerate
compact continuum of cluster points exactly when $b_{-1}$ lies outside
the closed unit disk. 
\item 
The conclusion of Corollary~\ref{c:main} hold also true if we
replace $S$ be the closed unit disk. In this case, both $A$ and
$B$ are \emph{countably infinite} unions of convex sets. 
In the following result, we show that a degenerate continuum
cannot occur as the set of cluster points when $A$ and $B$ are
\emph{finite} unions of nonempty closed convex sets. 
\end{enumerate}
\end{remark}

\begin{theorem}[finite unions of convex sets]
Suppose that $I$ and $J$ are nonempty finite index sets, 
let $(A_i)_{i\in I}$ and $(B_j)_{j\in J}$ be families of nonempty
closed convex subsets of a Euclidean space $X$, and set $A :=
\bigcup_{i\in I} A_i$ and $B := \bigcup_{j\in J} B_j$.
Consider a sequence of alternating projections $(a_n)_\nnn$ and
$(b_n)_\nnn$ generated by $A$ and $B$:
$b_{-1}\in X$, and $(\forall\nnn)$ $a_{n}\in P_Ab_{n-1}$ and
$b_n \in P_Ba_n$. 
Suppose that $(a_n)_\nnn$ and $(b_n)_\nnn$ are bounded, and that 
$b_n-a_n\to 0$ and $a_{n+1}-b_n\to 0$.
Then there exists a point $c\in A\cap B$ such that $a_n\to c$ and
$b_n\to c$. 
\end{theorem}
\begin{proof}
After relabeling and considering the tails of the sequences if 
necessary, 
we assume that each $A_i$ and each $B_j$ is projected upon
infinitely often.
The pigeonhole principle gives 
$(i_+,j_+)\in I\times J$ and subsequences 
$(a_{k_n})_\nnn$ and $(b_{k_n})_\nnn$ lying in $A_{i_+}$ and 
$B_{j_+}$
respectively. After passing to further subsequences if necessary,
we also assume that there is $c\in A_{i_+}\cap B_{j_+}$ such that 
$a_{k_n}\to c$ and $b_{k_n}\to c$. 
Set $I_{-} := \menge{i\in I}{c\notin A_i}$,
$I_+ := I\smallsetminus I_-$, 
$J_{-} := \menge{j\in J}{c\notin B_j}$, 
$J_+ := J\smallsetminus J_-$, 
$\delta := \min\{\min_{i\in I_-} d_{A_i}(c),
\min_{j\in J_-} d_{B_j}(c),1\}$, 
$A_- := \bigcup_{i\in I_-} A_i$, and
$B_- := \bigcup_{j\in J_-} B_j$.
Since $a_{k_n}\to c$, there exists $m\in\NN$ such that
$\|a_m-c\|<\delta/2$. 
Then 
$d_{B_-}(a_m)\geq d_{B_-}(c)-\|a_m-c\|
>\delta-\delta/2 = \delta/2>\|a_m-c\|\geq d_{B\smallsetminus
B_-}(a_m)$.
Hence 
$(\forall j\in J_-)$ $b_m\notin P_{B_j}(a_m)$ and similarly
$(\forall i\in I_-)$ $a_{m+1}\notin P_{A_i}(b_m)$.
Thus, $b_m\in \menge{P_{B_j}(a_m)}{j\in J_+}$ and
$a_{m+1}\in\menge{P_{A_i}(b_m)}{i\in I_+}$.
Therefore, because the projectors are nonexpansive, 
$\delta/2>\|a_m-c\|\geq\|b_m-c\|\geq\|a_{m+1}-c\| \geq \cdots$ and
recalling 
the assumption that all sets are projected upon
yields $I_-=J_-=\varnothing$, i.e., 
$c\in \bigcap_{i\in I} A_i\cap \bigcap_{j\in J}B_j$. 
Since $c$ is a cluster point of $(a_n)_\nnn$ and $(b_n)_\nnn$, it
thus follows that $\|a_n-c\|\to 0$ and $\|b_n-c\|\to 0$.
\end{proof}

\small

\section*{Acknowledgments}
The authors thank Dr.~Shawn Wang for helpful discussions.
HHB was partially supported by the Natural Sciences and
Engineering Research Council of 
Canada and by the Canada Research Chair
Program.
DN was supported by the research grant Technicom from Foundation
EADS.



\begin{thebibliography}{999}

\sepp 

\bibitem{BC2011}
H.H.\ Bauschke and P.L.\ Combettes,
\emph{Convex Analysis and Monotone Operator Theory in Hilbert Spaces},
Springer, 2011.

\bibitem{BLPW1}
H.H.\ Bauschke, D.R.\ Luke, H.M.\ Phan, and X.\ Wang,
Restricted normal cones and the method of alternating
projections: theory,
\emph{Set-Valued and Variational Analysis}, in press.

\bibitem{BLPW2}
H.H.\ Bauschke, D.R.\ Luke, H.M.\ Phan, and X.\ Wang,
Restricted normal cones and the method of alternating
projections: applications,
\emph{Set-Valued and Variational Analysis}, in press.

\bibitem{Bregman}
L.M.\ Bregman,
The method of successive project for finding a common point of
convex sets,
\emph{Soviet Mathematics Doklady}~6 (1965), 688--692. 

\bibitem{Ceg}
A.\ Cegielski,
\emph{Iterative Methods for Fixed Point Problems in Hilbert
Spaces},
Springer, 2012. 

\bibitem{CZ}
Y.\ Censor and S.A.\ Zenios,
\emph{Parallel Optimization},
Oxford University Press, 1997.

\bibitem{CT}
P.L.\ Combettes and H.J.\ Trussell,
Methods of successive projections for finding a common point of
sets in metric spaces,
\emph{Journal of Optimization Theory and Applications}~67(3)
(1990), 487--507. 


\bibitem{GK}
K.\ Goebel and W.A.\ Kirk,
\emph{Topics in Metric Fixed Point Theory},
Cambridge University Press, 1990.

\bibitem{GR}
K.\ Goebel and S.\ Reich,
\emph{Uniform Convexity, Hyperbolic Geometry, and Nonexpansive
Mappings}, Marcel Dekker, 1984. 



\end{thebibliography}
\end{document}